\documentclass[11pt,reqno]{amsart}
\usepackage[utf8]{inputenc}
\usepackage{amscd,amssymb,amsmath,amsthm}
\usepackage[arrow,matrix]{xy}
\usepackage{graphicx}
\usepackage{epstopdf}
\usepackage{color}
\usepackage{multicol}
\newtheorem{thm}{Theorem}
\newtheorem{defn}{Definition}
\newtheorem{lemma}{Lemma}
\newtheorem{pro}{Proposition}

\makeatletter
\@namedef{subjclassname@2020}{%
  \textup{2020} Mathematics Subject Classification}
\makeatother

\numberwithin{equation}{section} 

\begin{document}
\title [Stability analysis of a three-dimensional system of Topp model with diabetes]
{Stability analysis of a three-dimensional system of Topp model with diabetes}

\author{Z.S. Boxonov, U.A. Rozikov}

\address{ U.Rozikov$^{a,b,c}$\begin{itemize}
 \item[$^a$] V.I.Romanovskiy Institute of Mathematics, 9, University str. 100174, Tashkent, Uzbekistan;
\item[$^b$] New Uzbekistan University, 1, Movarounnahr str., 100000, Tashkent, Uzbekistan;
\item[$^c$] Faculty of Mathematics, National University of Uzbekistan, 4, University str. 100174, Tashkent, Uzbekistan.
\end{itemize}}
\email{rozikovu@yandex.ru}

\address{Z.Boxonov,
V.I.Romanovskiy Institute of Mathematics, 9, University str. 100174, Tashkent, Uzbekistan.}
 \email {z.boxonov@mathinst.uz}

\begin{abstract} Mathematical models of glucose, insulin, and pancreatic $\beta$-cell mass dynamics are essential for understanding the physiological basis of type 2 diabetes. This paper investigates the Topp model's discrete-time dynamics to represent these interactions. We perform a comprehensive analysis of the system's trajectory, examining both local and global behavior. First, we establish the invariance of the positive trajectory and analyze the existence of fixed points. Then, we conduct a complete stability analysis, determining the local and global asymptotic stability of these fixed points. Finally, numerical examples validate the effectiveness and applicability of our theoretical findings. Additionally, we provide biological interpretations of our results.
\end{abstract}

\subjclass[2020] {39A12, 39A30, 92C50}

\keywords{glucose, insulin, beta-cell, fixed point, local stability, global behavior.} \maketitle

\section{\bf Introduction}

To maintain normal human body functioning, it is essential to control the level of glucose in the blood within the range of $70-100 mg/dl$\footnote{https://en.wikipedia.org/wiki/Blood$_-$sugar$_-$level}.
Insulin, produced by beta-cells in the pancreas, facilitates the absorption of glucose by cells and plays a critical role in blood glucose regulation. High blood glucose levels trigger the release of insulin, which helps lower the concentration to a healthy level. As blood glucose levels decrease, insulin release gradually stops. This system, involving both insulin and glucagon, helps maintain blood glucose balance and prevent diabetes-related complications \cite{T, GP, YL}.

Mathematical models of glucose, insulin, and pancreatic beta-cell mass dynamics are crucial for understanding the physiological basis of type 2 diabetes development. Traditionally, type 2 diabetes was thought to arise from insulin insufficiency. The $GI\beta$ model developed by Topp and colleagues is a prominent model for its progression. Topp's model is foundational for studying diabetes progression \cite{T}. The Topp model is
\begin{equation}\label{system}
\left\{%
\begin{array}{lll}
    \frac{dG}{dt}=g_0-g_1G-cGI,\\[3mm]
    \frac{dI}{dt}=\frac{s_1 G^2}{s_2+G^2}\beta-kI, \\[3mm]
    \frac{d\beta}{dt}=(-d_0+r_1G-r_2G^2)\beta,
\end{array}%
\right.\end{equation}
where $G$ $(mg/dl)$, $I$ $(\mu U/ml)$, $\beta$ $(mg)$ stand for the plasma glucose concentration, insulin concentration and the mass of functional beta-cells (preserving appropriate insulin production
and secretion) at time $t$ (days), respectively. The parameter $g_0$ stands for the average rate of glucose infusion per day (with meal ingestion as a main source), including the hepatic glucose
production. The term $g_1G$ represents insulin-independent uptake of glucose, mainly by brain cells and nerve cells. In contrast, the term $cGI$ depicts the insulin-dependent uptake of glucose,
mostly by fat cells and muscle cells in the human body. In particular, the coefficient $c$ $(ml/\mu U/day)$ stands for insulin sensitivity. The insulin secretion from beta-cells is assumed to be triggered by increased glucose levels in the form of the Hill function with coefficient 2, and the parameter $s_1$ represents the secretory capacity per beta-cell. The insulin clearance rate is
denoted by $k$ (/day). The functional beta-cell mass is hypothesized to respond to glucose with a pattern similar to a downward parabola: moderate amount of glucose promotes the growth of
beta-cells, while a high glucose level exacerbates beta-cell apoptosis, resulting in the decrease of functional beta-cell mass and  $d_0$ $(d^{-1})$ is the death rate at zero glucose
and $r_1$ $(mg^{-1}\ dl\ d^{-1})$, $r_2$ $(mg^{-2}\ dl^2\ d^{-1})$ are constants \cite{T, YL}.

In this paper (as in \cite{Rpd}-\cite{B}) we study the discrete time dynamical systems associated to the system (\ref{system}).
In the equations (\ref{system}) we do the following replacements: $$G=x, I=y, \beta=z.$$
Define the operator $W:\mathbb{R}^{3}\rightarrow \mathbb{R}^{3}$ by
\begin{equation}\label{systema}
\left\{%
\begin{array}{lll}
    x'=g_0-g_1x-cxy+x,\\[3mm]
    y'=\frac{s_1 x^2}{s_2+x^2}z-ky+y,\\[3mm]
    z'=\left(-d_0+r_1x-r_2x^2\right)z+z,
\end{array}%
\right.\end{equation}
where all parameters in the model are positive. In the system (\ref{systema}), $x', y', z'$ means the next state relative to the initial state $x, y, z$ respectively.

We call the partition into types hereditary if for each possible state $\textbf{u}=(x,y,z)\in\mathbb{R}^{3}$ describing the current generation, the state $\textbf{u}'=(x',y',z')\in\mathbb{R}^{3}$ is uniquely defined describing the next generation. This means that the association $u\rightarrow u'$ defines a map $W:\mathbb{R}^{3}\rightarrow \mathbb{R}^{3}$ called the evolution operator \cite{L}, \cite{Rpd}.

The main problem for a given operator $W$ and arbitrarily initial point $\textbf{u}^{(0)}=(x^{(0)},y^{(0)},z^{(0)})\in\mathbb{R}^{3}$  is to describe the limit points of the trajectory $\{\textbf{u}^{(m)}\}_{m=0}^{\infty}$, where $$\textbf{u}^{(m)}=W(\textbf{u}^{m-1})=W^m(\textbf{u}^{(0)}).$$

\section{\bf Local Stability Analysis of Fixed Points}

Let $\mathbb{R}_{+}^{3}=\{(x,y,z): x,y,z\in\mathbb{R}, x\geq0, y\geq0, z\geq0\}.$

Note that the operator $W$ well defined on $\mathbb{R}^3$. But to define a dynamical system of continuous operator as glucose, insulin, and pancreatic beta-cell mass  we assume $x\geq0$, $y\geq0$ and $z\geq0$. Therefore, we choose parameters of the operator $W$ to guarantee that it maps $\mathbb{R}^3_{+}$ to itself.


We denote
$$\Omega=\left\{(x,y,z)\in\mathbb{R}_{+}^{3}: g_0\leq x\leq A, 0\leq y\leq B, 0\leq z\leq C\right\},$$
where $$A=\frac{r_1+\sqrt{r_1^2+4r_2(1-d_0)}}{2r_2}, \ B=\frac{1-g_1}{c}, \ C=\frac{(1-g_1)k}{s_1c}.$$

\begin{lemma} If
\begin{equation}\label{parametr}
0<g_1<1,\ 0<k\leq1,\ 0<d_0\leq1,\ g_0/g_1\leq A,\ r_1^2\leq4r_2d_0
\end{equation}
then the operator (\ref{systema}) maps the set $\Omega$ to itself.
\end{lemma}
\begin{proof}
Let for any $(x,y,z)\in\Omega$, i.e., $g_0\leq x\leq A, 0\leq y\leq B, 0\leq z\leq C.$ Then
$$x'=g_0+x(1-g_1-cy)\geq g_0+x(1-g_1-c\cdot B)=g_0+x(1-g_1-c\cdot\frac{1-g_1}{c})= g_0,$$
$$x'=g_0+x(1-g_1-cy)\leq g_0+x(1-g_1)\leq g_0+A(1-g_1)=A-g_1(A-g_0/g_1)\leq A.$$
Thus $g_0\leq x'\leq A.$

Since $0<k\leq1$, it is clear that $y'\geq0.$ We will now show that $y'\leq B.$
$$y'=\frac{s_1 x^2}{s_2+x^2}z-ky+y\leq s_1\cdot C-ky+B= s_1\cdot\frac{(1-g_1)k}{s_1c}-ky+B=B-k(B-y)\leq B.$$
It follows that $0\leq y'\leq B.$

It can be seen that the number $A$ is a positive solution to the equation $1-d_0+r_1x-r_2x^2=0$. When $x$ ranges from $0$ to $A$, the expression $1-d_0+r_1x-r_2x^2$ is always non-negative. Since $[g_0,A]\subset[0, A]$, then $z'\geq0.$
$$z'=\left(1-d_0+r_1x-r_2x^2\right)z\leq \left(1-d_0+\frac{r_1^2}{2r_2}-\frac{r_1^2}{4r_2}\right)C=\left(1-\frac{4r_2d_0-r_1^2}{4r_2}\right)C\leq C.$$
It follows that $z'$ also changes from 0 to $C$. Thus $(x',y',z')\in\Omega.$
\end{proof}

\subsection{Fixed points.}\
First, we discuss the existence of the fixed points.
A point $u\in\Omega$ is called a fixed point of $W$ if $W(u)=u$.

\begin{pro}\label{fixed} The fixed points for (\ref{systema}) are as follows:
\begin{itemize}
  \item If $r_1^2<4r_2d_0$ then the operator (\ref{systema}) has a unique fixed point $u^*=\left(\frac{g_0}{g_1},0,0\right).$
  \item If $r_1^2=4r_2d_0$, $\frac{g_1r_1}{2r_2}< g_0< \frac{r_1}{2r_2}$, $s_2r_2(2r_2g_0-g_1r_1)\leq d_0(r_1-2r_2g_0)$ then mapping (\ref{systema}) has two fixed points with $$u^*_{1}=\left(\frac{g_0}{g_1},0,0\right), \ \ u^*_{2}=\left(x^*,y^*,z^*\right),$$
\end{itemize}
\end{pro}
where
\begin{equation}\label{x*y*z*}
x^*=\frac{r_1}{2r_2},\ y^*=\frac{g_0-g_1x^*}{cx^*},\ z^*=\frac{ky^*(s_2+x^*{^2})}{s_1x^*{^2}}.
\end{equation}
\begin{proof} The equation $W(u)=u$ is the following system
\begin{equation}\label{fsystema}
\left\{%
\begin{array}{lll}
    x=g_0+x(1-g_1-cy),\\[3mm]
    y=\frac{s_1 x^2}{s_2+x^2}z+(1-k)y,\\[3mm]
    z=\left(1-d_0+r_1x-r_2x^2\right)z.
\end{array}%
\right.\end{equation}
The third equation of system (\ref{fsystema}) shows that either $z=0$ or $r_2x^2-r_1x+d_0=0.$ It is easy to see that $x=\frac{g_0}{g_1}, y=0, z=0$ and $x^*=\frac{r_1}{2r_2}, y^*=\frac{g_0-g_1x^*}{cx^*}, z^*=\frac{ky^*(s_2+x^*{^2})}{s_1x^*{^2}}$ (under conditions (\ref{parametr})) are solution to (\ref{fsystema}). Now let's find additional conditions for the parameters so that the point $(x^*,y^*,z^*)$ belongs to the set $\Omega$.

If $r_1^2<4r_2d_0$ then $x^*\notin[g_0, A],$ if $r_1^2=4r_2d_0$ and $2r_2g_0<r_1$ then $x^*\in[g_0, A].$ It follows that $y\leq B.$ Indeed,
$$y^*=\frac{g_0-g_1x^*}{cx^*}\leq\frac{x^*-g_1x^*}{cx^*}=\frac{1-g_1}{c}=B.$$
If $g_1r_1<2r_2g_0$ then $y^*\geq0.$, i.e., $y^*\in[0,B].$ Since $x$ and $y$ are non-negative, it follows that $z$ is also non-negative.
We solve the inequality $z^*\leq C$ and form the condition $s_2r_2(2r_2g_0-g_1r_1)\leq d_0(r_1-2r_2g_0)$ for the parameters.

So, if the parameters satisfy the following conditions
\begin{equation}\label{2fixmavjud}
r_1^2=4r_2d_0,\ \frac{g_1r_1}{2r_2}<g_0<\frac{r_1}{2r_2},\ s_2\leq\frac{d_0(r_1-2r_2g_0)}{r_2(2r_2g_0-g_1r_1)},
\end{equation}
then the point $(x^*,y^*,z^*)$ belongs to the set $\Omega$.
\end{proof}

\subsection{Types of the fixed points}\

Now we shall examine the type of the fixed points.
\begin{defn}\label{d1}
(see \cite{D}) A fixed point $u^*$ of the operator $W$ is called
hyperbolic if its Jacobian $J$ at $u^*$ has no eigenvalues on the
unit circle.
\end{defn}

\begin{defn}\label{d2}
(see \cite{D}) A hyperbolic fixed point $u^*$ called:

1) attracting if all the eigenvalues of the Jacobi matrix $J(u^*)$
are less than 1 in absolute value;

2) repelling if all the eigenvalues of the Jacobi matrix $J(u^*)$
are greater than 1 in absolute value;

3) a saddle otherwise.
\end{defn}

Before analyzing the fixed points we give the following useful lemma (\cite{ChX}).
\begin{lemma}\label{F(l)}
Let $F(\lambda)=\lambda^2+B^*\lambda+C^*,$ where $B^*$ and $C^*$ are two real constants. Suppose $\lambda_1$ and $\lambda_2$ are two roots of $F(\lambda)=0$. Then the following statements hold.
\begin{enumerate}
  \item[(i)] If $F(1)>0$ then
  \item[(i.1)] $|\lambda_1|<1$ and $|\lambda_2|<1$ if and only if $F(-1)>0$ and $C^*<1;$
  \item[(i.2)] $\lambda_1=-1$ and $\lambda_2\neq-1$ if and only if $F(-1)=0$ and $B^*\neq2;$
  \item[(i.3)] $|\lambda_1|<1$ and $|\lambda_2|>1$ if and only if $F(-1)<0;$
  \item[(i.4)] $|\lambda_1|>1$ and $|\lambda_2|>1$ if and only if $F(-1)>0$ and $C^*>1;$
  \item[(i.5)] $\lambda_1$ and $\lambda_2$ are a pair of conjugate complex roots and $|\lambda_1|=|\lambda_2|=1$ if only if $-2<B^*<2$ and $C^*=1;$
  \item[(i.6)] $\lambda_1=\lambda_2=-1$ if only if $F(-1)=0$ and $B=2.$
  \item[(ii)] If $F(1)=0,$ namely, 1 is one root of $F(\lambda)=0,$ then the other root $\lambda$ satisfies $|\lambda|=(<, >)1$  if and only if $|C^*|=(<, >)1.$
  \item[(iii)] If $F(1)<0$ then $F(\lambda)=0$ has one root lying in $(1;\infty).$ Moreover,
  \item[(iii.1)] the other root $\lambda$ satisfies $\lambda<(=)-1$ if and only if $F(-1)<(=)0;$
  \item[(iii.2)] the other root $\lambda$ satisfies $-1<\lambda<1$ if and only if $F(-1)>0.$
\end{enumerate}
\end{lemma}

To find the type of a fixed point of the operator (\ref{systema})
we write the Jacobi matrix:

$$J(u)=J_{W}=\left(%
\begin{array}{ccc}
  1-g_1-cy & -cx & 0 \\
  \frac{2s_1s_2x}{(s_2+x^2)^2}z & 1-k & \frac{s_1x^2}{s_2+x^2} \\
  (r_1-2r_2x)z & 0 & 1-d_0+r_1x-r_2x^2 \\
\end{array}%
\right).$$
The eigenvalues of the Jacobi matrix at the fixed point $u_1^*=\left(\frac{g_0}{g_1},0, 0\right)$ are as follows
$$\lambda_1=1-g_1,\ \lambda_2=1-k,\ \lambda_3=1-d_0+\frac{r_1g_0}{g_1}-\frac{r_2g_0^2}{g^2_1}.$$
By (\ref{parametr}) we have $0<\lambda_1<1, 0\leq\lambda_2<1$. We write $\lambda_3$ as follows.
$$\lambda_3=1-d_0+\frac{r_1g_0}{g_1}-\frac{r_2g_0^2}{g^2_1}=1-r_2\left(\frac{g_0}{g_1}-\frac{r_1}{2r_2}\right)^2+\frac{r^2_1-4r_2d_0}{4r_2}.$$
If $r^2_1\leq4r_2d_0$ then $\lambda_3<1$.

We calculate eigenvalues of Jacobian matrix at the fixed point $u_2^*=\left(x^*,y^*, z^*\right)$.
The characteristic equation is
\begin{equation}\label{F(L)}
(\lambda-1)F(\lambda)=(\lambda^2+B^*\lambda+C^*)(\lambda-1),
\end{equation}
where
\begin{center}
$B^*=k+\frac{g_0}{x^*}-2,$  \\[1mm]
$C^*=\left(1-\frac{g_0}{x^*}\right)(1-k)+\frac{2s_2kcy^*}{s_2+x^*{^2}}.$
\end{center}
From (\ref{F(L)}) it is clear that one of the eigenvalues is equal to one. Let us determine the sign of $F(1)$, $F(-1)$ and $C^*-1$ under conditions (\ref{parametr}), (\ref{2fixmavjud}).
\begin{center}
$F(1)=1+B^*+C^*=1+k+\frac{g_0}{x^*}-2+\left(1-\frac{g_0}{x^*}\right)(1-k)+\frac{2s_2kcy^*}{s_2+x^*{^2}}=\frac{g_0k}{x^*}+\frac{2s_2kcy^*}{s_2+x^*{^2}}>0,$ \\[2mm]
$F(-1)=1-B^*+C^*=1-k-\frac{g_0}{x^*}+2+\left(1-\frac{g_0}{x^*}\right)(1-k)+\frac{2s_2kcy^*}{s_2+x^*{^2}}=\left(2-\frac{g_0}{x^*}\right)(2-k)+\frac{2s_2kcy^*}{s_2+x^*{^2}}>0,$ \\[2mm]
$C^*-1=\left(1-\frac{g_0}{x^*}\right)(1-k)+\frac{2s_2kcy^*}{s_2+x^*{^2}}-1=\frac{2s_2k(g_0-g_1x^*)}{x^*(s_2+x^*{^2})}-\frac{g_0}{x^*}\left(1-k\right)-k=$ \\[2mm]
$=\frac{1}{x^*(s_2+x^*{^2})}\cdot\left(2s_2k(g_0-g_1x^*)-g_0(1-k)(s_2+x^*{^2})-kx^*(s_2+x^*{^2})\right)\leq$ \\[2mm]
$\leq\frac{k}{x^*(s_2+x^*{^2})}\cdot\left(2s_2(g_0-g_1x^*)-x^*(s_2+x^*{^2})\right)=$ \\[2mm]
$=\frac{k}{x^*(s_2+x^*{^2})}\cdot\left(s_2(2(g_0-g_1x^*)-x^*)-x^*{^3}\right)\leq$ \\[2mm]
$\leq\frac{k}{x^*(s_2+x^*{^2})}\cdot\left(\frac{x^*{^2}(x^*-g_0)}{g_0-g_1x^*}\cdot\left(2(g_0-g_1x^*)-x^*\right)-x^*{^2}\right)=$ \\[2mm]
$=\frac{kx^*}{(s_2+x^*{^2})(g_0-g_1x^*)}\cdot\left(-(1+g_1)x^*{^2}+2g_0(1+g_1)x^*-2g_0^2\right)=$ \\[2mm]
$=-\frac{k(1+g_1)x^*}{(s_2+x^*{^2})(g_0-g_1x^*)}\cdot\left((x^*-g_0)^2+\frac{(1-g_1)g_0^2}{1+g_1}\right)<0.$ \\[2mm]
\end{center}
According to item (i.1) of Lemma \ref{F(l)}, $|\lambda_{2,3}|<1.$

Thus the type of fixed points the following theorem holds.
\begin{thm}\label{type} The type of the fixed points for (\ref{systema}) are as follows:
\begin{itemize}
  \item[i)] if $r_1^2<4r_2d_0$, then the unique fixed point $u^*$ of the operator (\ref{systema}) is the attracting.
  \item[ii)] if $r_1^2=4r_2d_0$, $\frac{g_1r_1}{2r_2}< g_0< \frac{r_1}{2r_2}$, $s_2r_2(2r_2g_0-g_1r_1)\leq d_0(r_1-2r_2g_0)$ then the operator (\ref{systema}) has two fixed points $u_1^*, u_2^*$ and the point $u_1^*$ is attracting, the point $u_2^*$ is non-hyperbolic (but, $u_2^*$ is semi-attracting \footnote{means that two eigenvalues are less than 1 in absolute value.}).
  \end{itemize}
\end{thm}

\section{\bf Periodic  points.}\

A point $\textbf{u}$ in $\Omega$ is called periodic point of $W$ if there exists $p\in \mathbb N$ so that $W^{p}(\textbf{u})=\textbf{u}$. The smallest positive integer $p$ satisfying $W^{p}(\textbf{u})=\textbf{u}$ is called the prime period or least period of the point $\textbf{u}.$

\begin{thm}\label{perT} For $p\geq 2$ the operator (\ref{systema}) does not have any $p$-periodic point in the set $\Omega.$
\end{thm}
\begin{proof} Let's consider the following system.
\begin{equation}\label{Psystema}
\left\{%
\begin{array}{lll}
    x=x^{(p)}=g_0+x^{(p-1)}\left(1-g_1-cy^{(p-1)}\right),\\[3mm]
    y=y^{(p)}=\frac{s_1 \left(x^{(p-1)}\right){^2}}{s_2+\left(x^{(p-1)}\right){^2}}z^{(p-1)}+(1-k)y^{(p-1)},\\[3mm]
    z=z^{(p)}=\left(1-d_0+r_1x^{(p-1)}-r_2\left(x^{(p-1)}\right){^2}\right)z^{(p-1)}.
\end{array}%
\right.\end{equation}
From the third equation of system (\ref{Psystema}) we have
\begin{center}
$z=z^{(p-1)}\left(1-r_2\left(x^{(p-1)}-\frac{r_1}{2r_2}\right)^2+\frac{r_1^2-4r_2d_0}{4r_2}\right)
=z^{(p-2)}\left(1-r_2\left(x^{(p-1)}-\frac{r_1}{2r_2}\right)^2+\frac{r_1^2-4r_2d_0}{4r_2}\right)\left(1-r_2\left(x^{(p-2)}-\frac{r_1}{2r_2}\right)^2+\frac{r_1^2-4r_2d_0}{4r_2}\right)
=...=z\left(1-r_2\left(x^{(p-1)}-\frac{r_1}{2r_2}\right)^2+\frac{r_1^2-4r_2d_0}{4r_2}\right)\left(1-r_2\left(x^{(p-2)}-\frac{r_1}{2r_2}\right)^2+\frac{r_1^2-4r_2d_0}{4r_2}\right)\cdot...
\cdot\left(1-r_2\left(x-\frac{r_1}{2r_2}\right)^2+\frac{r_1^2-4r_2d_0}{4r_2}\right).$
\end{center}
Since $z\neq0$, we have
\begin{equation}\label{zP}
\begin{array}{ll}
\left(1-r_2\left(x^{(p-1)}-\frac{r_1}{2r_2}\right)^2+\frac{r_1^2-4r_2d_0}{4r_2}\right)\left(1-r_2\left(x^{(p-2)}-\frac{r_1}{2r_2}\right)^2+\frac{r_1^2-4r_2d_0}{4r_2}\right)\cdot...\\
...\cdot\left(1-r_2\left(x-\frac{r_1}{2r_2}\right)^2+\frac{r_1^2-4r_2d_0}{4r_2}\right)=1
\end{array}%
\end{equation}
Under the condition $r_1^2<4r_2d_0$, the left side of equation (\ref{zP}) will always be less than 1.
In this case, the equation will not have roots, which means that operator (\ref{systema}) does not have a $p$-periodic point in the set $\Omega$.

Let $r_1^2=4r_2d_0.$ Then, for equation (\ref{zP}) to have a root, it is necessary and sufficient to satisfy the relations
\begin{center}
$x^{(p-1)}=x^{(p-2)}=...=x=\frac{r_1}{2r_2}.$
\end{center}
These relationships hold only at a fixed point $\frac{r_1}{2r_2}.$ So, in this case, operator (\ref{systema}) does not have $p$-periodic points in the set $\Omega$.
\end{proof}

\section{\bf Global Behavior}\
In this section for any initial point $\left(x^{(0)}, y^{(0)}, z^{(0)}\right)\in \Omega$ we investigate behavior of the trajectories $\left(x^{(n)},y^{(n)},z^{(n)}\right)=W^n\left(x^{(0)}, y^{(0)},z^{(0)}\right), n\geq1.$

We have
\begin{equation}\label{xnynzn}
\begin{array}{lll}
    x^{(n)}=g_0+x^{(n-1)}\left(1-g_1-cy^{(n-1)}\right),\\[3mm]
    y^{(n)}=\frac{s_1 \left(x^{(n-1)}\right){^2}}{s_2+\left(x^{(n-1)}\right){^2}}z^{(n-1)}+(1-k)y^{(n-1)},\\[3mm]
    z^{(n)}=\left(1-d_0+r_1x^{(n-1)}-r_2\left(x^{(n-1)}\right){^2}\right)z^{(n-1)}.
\end{array}%
\end{equation}
\begin{lemma}\label{lemmaconverges} Assume that (\ref{parametr}) holds. For sequences $x^{(n)}$, $y^{(n)}$ and $z^{(n)}$, the following property holds:
\begin{itemize}
  \item [\textbf{(i)}] if $r_1^2<4r_2d_0$, then the sequence $z^{(n)}$ is monotonically decreasing and converges to zero;
  \item [\textbf{(ii)}] if the sequence $z^{(n)}$ converges to zero, then the sequence $y^{(n)}$ also has a limit and converges to zero;
  \item [\textbf{(iii)}] if the sequence $y^{(n)}$ converges to zero, then the sequence $x^{(n)}$ also has a limit and converges to $\frac{g_0}{g_1}$.
\end{itemize}
\end{lemma}
\begin{proof}
First, we prove the assertion \textbf{(i)}. From third equation of system (\ref{xnynzn}) we get
\begin{center}
  $\frac{z^{(n)}}{z^{(n-1)}}=1-d_0+r_1x^{(n-1)}-r_2\left(x^{(n-1)}\right){^2}=1-r_2\left(x^{(n-1)}-\frac{r_1}{2r_2}\right)^2+\frac{r_1^2-4r_2d_0}{4r_2}<1$.
\end{center}
This implies that $z^{(n)}$ is a decreasing sequence.
\begin{center}
  $z^{(n)}=\left(1-d_0+r_1x^{(n-1)}-r_2\left(x^{(n-1)}\right){^2}\right)z^{(n-1)}=\left(1-r_2\left(x^{(n-1)}-\frac{r_1}{2r_2}\right)^2+\frac{r_1^2-4r_2d_0}{4r_2}\right)z^{(n-1)}<
  \left(1+\frac{r_1^2-4r_2d_0}{4r_2}\right)z^{(n-1)}<\left(1+\frac{r_1^2-4r_2d_0}{4r_2}\right)^2z^{(n-2)}<...<\left(1+\frac{r_1^2-4r_2d_0}{4r_2}\right)^nz^{(0)}.$
\end{center}
Thus $0\leq z^{(n)}<\left(1+\frac{r_1^2-4r_2d_0}{4r_2}\right)^nz^{(0)}.$ Consequently $$\lim\limits_{n\to \infty}z^{(n)}=0.$$

Let's prove the assertion \textbf{(ii)}. By analyzing the second equation of (\ref{xnynzn}) and considering the boundedness of the sequence $x^{(n)}$, we have
\begin{equation}
\frac{s_1g_0^2}{s_2+g_0^2}z^{(n-1)}\leq y^{(n)}-(1-k)y^{(n-1)}\leq\frac{s_1A^2}{s_2+A^2}z^{(n-1)}.
\end{equation}
Since the sequence $z^{(n)}$ converges to zero,  it follows that the sequence $y^{(n)}-(1-k)y^{(n-1)}$ also converges to zero.

Since $y^{(n)}$ is bounded, it has a well-defined upper limit, denoted by $\alpha$. Then there must exist a subsequence $y^{(n_j)}$ that converges to $\alpha$. Furthermore, the fact that $y^{(n)}-(1-k)y^{(n-1)}$ converges to zero implies another subsequence $y^{(n_{j}-1)}$ approaches $\frac{\alpha}{1-k}$. Since $k$ is between 0 and 1, $\frac{\alpha}{1-k}$ must be greater than or equal to $\alpha$.
However, since $\alpha$ is an upper limit, it follows that $\frac{\alpha}{1-k}\leq\alpha.$
Therefore, $\frac{\alpha}{1-k}$ can only equal $\alpha$, which implies $\alpha=0$. We conclude that the sequence $y^{(n)}$ has a limit, which is equal to zero.

\textbf{(iii)}. Using the first equation of (\ref{xnynzn}) and the fact that $x^{(n)}$ is bounded, we can conclude that
\begin{equation}\label{xne}
g_0-cAy^{(n-1)}\leq x^{(n)}-(1-g_1)x^{(n-1)}\leq g_0-cg_0y^{(n-1)}.
\end{equation}
Since $y^{(n)}$ converges to zero, it follows from (\ref{xne}) that $$\lim\limits_{n\to \infty}\left(x^{(n)}-(1-g_1)x^{(n-1)}\right)=g_0.$$
Since $x^{(n)}$ is bounded, it possesses upper and lower limits, denoted by $\alpha$ and $\beta$, respectively. To arrive at a contradiction, suppose $x^{(n)}$ has no limit (i.e.,$\alpha\neq\beta$). This would necessitate the existence of two subsequences: $x^{(n_i)}$ converging to $\alpha$ and $x^{(n_j)}$ converging to $\beta$.
$$\lim\limits_{i\to \infty}\left(x^{(n_i)}-(1-g_1)x^{(n_i-1)}\right)=g_0,\ \Rightarrow \ \lim\limits_{i\to \infty}x^{(n_i-1)}=\frac{\alpha-g_0}{1-g_1}\leq\alpha, \ \Rightarrow \ \alpha\leq\frac{g_0}{g_1},$$
$$\lim\limits_{i\to \infty}\left(x^{(n_i+1)}-(1-g_1)x^{(n_i)}\right)=g_0,\ \Rightarrow \ \lim\limits_{i\to \infty}x^{(n_i+1)}=g_0+(1-g_1)\alpha\leq\alpha,\ \Rightarrow \ \alpha\geq\frac{g_0}{g_1},$$
$$\lim\limits_{j\to \infty}\left(x^{(n_j)}-(1-g_1)x^{(n_j-1)}\right)=g_0,\ \Rightarrow \ \lim\limits_{j\to \infty}x^{(n_j-1)}=\frac{\beta-g_0}{1-g_1}\geq\beta,\ \Rightarrow \ \beta\geq\frac{g_0}{g_1},$$
$$\lim\limits_{j\to \infty}\left(x^{(n_j+1)}-(1-g_1)x^{(n_j)}\right)=g_0,\ \Rightarrow \ \lim\limits_{j\to \infty}x^{(n_j+1)}=g_0+(1-g_1)\beta\geq\beta,\ \Rightarrow \ \beta\leq\frac{g_0}{g_1},$$
Therefore, we conclude that $\alpha=\beta=\frac{g_0}{g_1}$ i.e., $\lim\limits_{n\to \infty}x^{(n)}=\frac{g_0}{g_1}$.
\end{proof}

\begin{thm}\label{Omega} Assume that (\ref{parametr}) holds and let $r_1^2<4r_2d_0$. Then the trajectory converges to fixed point $u^*$ for any initial point $\left(x^{(0)}, y^{(0)},z^{(0)}\right)\in\Omega$, i.e.,
$$\lim\limits_{n\to \infty}W^n\left(x^{(0)},y^{(0)},z^{(0)}\right)=\left(\frac{g_0}{g_1},0,0\right).$$

\end{thm}
\begin{proof} The proof follows from Lemma \ref{lemmaconverges}.
\end{proof}

A set $S$ is called invariant with respect to $W$ if $W(S)\subset S.$

Denote
\begin{center}
$\Omega_1=\left\{(x,y,z)\in \Omega: x\leq\frac{g_0}{g_1}\right\}.$\\[1mm]
$\Omega_2=\{(x,y,z)\in \Omega_1: z<z^*\}.$
\end{center}

\begin{pro}\label{invariant} The sets $\Omega_1$ and $\Omega_2$ are invariant with respect to the operator $W$.
\end{pro}
\begin{proof}

\textbf{(1)} Let $(x,y,z)\in\Omega_1$, i.e., $g_0\leq x\leq\frac{g_0}{g_1}.$ Then
\begin{center}
$x'-\frac{g_0}{g_1}=g_0+x\left(1-g_1-cxy\right)-\frac{g_0}{g_1}\leq g_0+\frac{g_0}{g_1}\left(1-g_1-cxy\right)-\frac{g_0}{g_1}=-\frac{g_0cxy}{g_1}\leq0.$
\end{center}

Thus $(x',y',z')\in\Omega_1$, i.e., $W(\Omega_1)\subset\Omega_1.$

\textbf{(2)} Let $(x,y,z)\in\Omega_2$, i.e., $z<z^*.$ Then
\begin{center}
$z'=z\left(1-d_0+r_1x-r_2x^2\right)=z\left(1-r_2(x-\frac{r_1}{2r_2})^2+\frac{r_1^2-4r_2d_0}{4r_2}\right)\leq z<z^*.$
\end{center}

So $(x',y',z')\in\Omega_2$, i.e., $W(\Omega_2)\subset\Omega_2.$

\end{proof}

Let $r_1^2=4r_2d_0$. The following theorem  gives full description of the set of limit points for the trajectory of any initial point $\left(x^{(0)}, y^{(0)},z^{(0)}\right)$ in the set $\Omega$.
\begin{thm}\label{Omega-}
Assume that (\ref{parametr}) and (\ref{2fixmavjud}) holds. For the operator $W$ given by (\ref{systema}) the following hold:
\begin{itemize}
  \item [\textbf{(i)}] If $x^{(n)}>\frac{g_0}{g_1}$ for any natural number $n$, then the trajectory of the initial point $\left(x^{(0)}, y^{(0)},z^{(0)}\right)$ taken from the set $\Omega\setminus\Omega_1$, is equal to the limit point $u^*_1$.
  \item [\textbf{(ii)}] The trajectory of any initial point $\left(x^{(0)}, y^{(0)},z^{(0)}\right)$ in the set $\Omega_2$ converges to fixed point $u^*_1$.
  \item [\textbf{(iii)}] If $z^{(n)}>z^*$ for any natural number $n$, then the trajectory of the initial point $\left(x^{(0)}, y^{(0)},z^{(0)}\right)$ in the set $\Omega_1\setminus\Omega_2$, is equal to the limit point $u^*_2$.
\end{itemize}
\end{thm}
\begin{proof} We have
\begin{equation}\label{xnynzn+}
\begin{array}{lll}
    x^{(n)}=g_0+x^{(n-1)}\left(1-g_1-cy^{(n-1)}\right),\\[3mm]
    y^{(n)}=\frac{s_1 \left(x^{(n-1)}\right){^2}}{s_2+\left(x^{(n-1)}\right){^2}}z^{(n-1)}+(1-k)y^{(n-1)},\\[3mm]
    z^{(n)}=-r_2\left(x^{(n-1)}-\frac{r_1}{2r_2}\right)^2z^{(n-1)}+z^{(n-1)}.
\end{array}%
\end{equation}
First, we prove the assertion \textbf{(i)}. Let all values of $x^{(n)}$ are greater than $\frac{g_0}{g_1}$. Then
\begin{center}
$x^{(n)}-x^{(n-1)}=g_0-g_1x^{(n-1)}-cx^{(n-1)}y^{(n-1)}=g_1\left(\frac{g_0}{g_1}-x^{(n-1)}\right)-cx^{(n-1)}y^{(n-1)}\leq0.$\\[1mm]
$z^{(n)}-z^{(n-1)}=-r_2\left(x^{(n-1)}-\frac{r_1}{2r_2}\right)^2z^{(n-1)}\leq0.$
\end{center}
Therefore, both sequences $x^{(n)}$ and $z^{(n)}$ are decreasing. Since both sequences $x^{(n)}$ and $z^{(n)}$ are decreasing and bounded from below, we have:
\begin{equation}\label{xn>}
\lim\limits_{n\to \infty}x^{(n)}\geq\frac{g_0}{g_1},\
\lim\limits_{n\to \infty}z^{(n)}\geq 0.
\end{equation}
We estimate $x^{(n)}$ and $z^{(n)}$ by the following:
\begin{center}
$x^{(n)}=g_0-g_1x^{(n-1)}-cx^{(n-1)}y^{(n-1)}+x^{(n-1)}< g_0+(1-g_1)x^{(n-1)}$ \\[1mm]
$< g_0+(1-g_1)\left(g_0+(1-g_1)x^{(n-2)}\right)$\\[1mm]
$< g_0+g_0(1-g_1)+...+g_0(1-g_1)^{n-1}+(1-g_1)^{n}x^{(0)}$\\[1mm]
$=\frac{g_0}{g_1}\left(1-(1-g_1)^n\right)+(1-g_1)^{n}x^{(0)}.$\\[1mm]
$z^{(n)}=\left(1-r_2\left(x^{(n-1)}-\frac{r_1}{2r_2}\right)^2\right)z^{(n-1)}<\left(1-r_2\left(\frac{g_0}{g_1}-\frac{r_1}{2r_2}\right)^2\right)z^{(n-1)}$\\
$<\left(1-r_2\left(\frac{g_0}{g_1}-\frac{r_1}{2r_2}\right)^2\right)^2z^{(n-2)}<...<\left(1-r_2\left(\frac{g_0}{g_1}-\frac{r_1}{2r_2}\right)^2\right)^nz^{(0)}.$
\end{center}
Thus  $x^{(n)}<\frac{g_0}{g_1}\left(1-(1-g_1)^n\right)+(1-g_1)^{n}x^{(0)},\ z^{(n)}<\left(1-r_2\left(\frac{g_0}{g_1}-\frac{r_1}{2r_2}\right)^2\right)^nz^{(0)}.$ Consequently
\begin{equation}\label{xn<}\lim\limits_{n\to \infty}x^{(n)}\leq\frac{g_0}{g_1}, \ \lim\limits_{n\to \infty}z^{(n)}\leq 0.
\end{equation}
Based on the inequalities (\ref{xn>}) and (\ref{xn<}), we can conclude that the sequence $x^{(n)}$ converges to  $\frac{g_0}{g_1}$, and the sequence $z^{(n)}$ converges to 0, respectively.
From (\ref{xnynzn+}) it follows $\lim\limits_{n\to \infty}y^{(n)}=0.$

Let's prove the assertion \textbf{(ii)}. Let $\left(x^{(0)}, y^{(0)},z^{(0)}\right)\in \Omega_2.$ Since $\Omega_2$ is an invariant set, then $g_0\leq x^{(n)}\leq \frac{g_0}{g_1}$, $0\leq y^{(n)}\leq C$, $0\leq z^{(n)}<z^*$ for any natural number $n.$ Moreover, since $z^{(n)}$ is strictly decreasing, $0\leq\lim\limits_{n\to \infty}z^{(n)}<z^*.$ Suppose $\lim\limits_{n\to \infty}z^{(n)}=\overline{z}\neq0$. From the third equation of system (\ref{xnynzn+}) it is clear that the existence of the limit $z^{(n)}$ implies the existence of the limit $x^{(n)}$. Similarly, from the first equation of system (\ref{xnynzn+}), the existence of the limit $x^{(n)}$ implies the existence of the limit $y^{(n)}$. We denote the limits of the sequences $x^{(n)}$ and $y^{(n)}$ by $\overline{x}$ and $\overline{y}$, respectively. Then from the system (\ref{xnynzn+}), we form the following:
\begin{equation}\label{x-}
\left\{%
\begin{array}{lll}
    \overline{x}=g_0+\overline{x}(1-g_1-c\overline{y}),\\[1mm]
    \overline{y}=\frac{s_1 \overline{x}{^2}}{s_2+\overline{x}{^2}}\overline{z}+(1-k)\overline{y},\\[1mm]
    \overline{z}=\left(1-r_2\left(\overline{x}-\frac{r_1}{2r_2}\right)^2\right)\overline{z}.
\end{array}%
\right.\end{equation}
By (\ref{x-}) we have
\begin{center}
  $\overline{x}=\frac{r_1}{2r_2}=x^*, \ \overline{y}=y^*,\ \overline{z}=z^*$   \ (see (\ref{x*y*z*})).
\end{center}
From this contradiction it follows that the sequence $z^{(n)}$ converges to 0. Due to the convergence of $z^{(n)}$ to zero, Lemma \ref{lemmaconverges} (second part)  guarantees that $y^{(n)}$ also converges to zero. Leveraging this result and Lemma \ref{lemmaconverges} (third part), we can further conclude that $x^{(n)}$ converges to $g_0/g_1$.

\textbf{(iii)}. Let all values of $z^{(n)}$ are greater than $z^*$. Then, since $z^{(n)}$ is decreasing and bounded from below, $\lim\limits_{n\to \infty}z^{(n)}\geq z^*.$ From the third equation of system (\ref{xnynzn+}) it is clear that the existence of the limit $z^{(n)}$ implies the existence of the limit $x^{(n)}$. Similarly, from the first equation of system (\ref{xnynzn+}), the existence of the limit $x^{(n)}$ implies the existence of the limit $y^{(n)}$. By (\ref{xnynzn+}) we obtain
\begin{center}
$\lim\limits_{n\to \infty}x^{(n)}=x^*, \ \lim\limits_{n\to \infty}y^{(n)}=y^*, \ \lim\limits_{n\to \infty}z^{(n)}=z^*.$
\end{center}
\end{proof}

%

\begin{center}
  \begin{figure} [h!]
    \centering
    \includegraphics[width=0.8\textwidth]{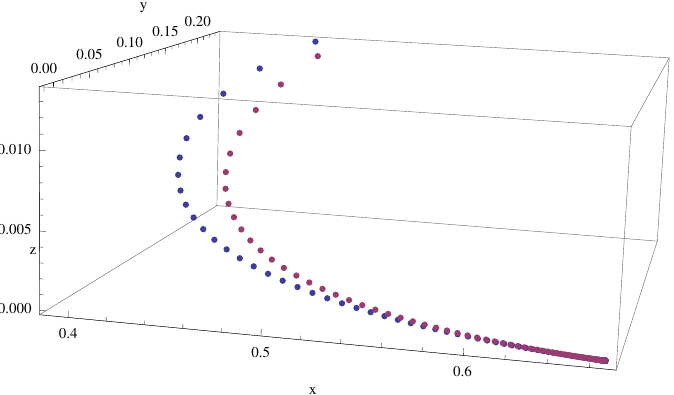}
    \caption{\footnotesize{Considering the system defined by equation (\ref{systema}) with the following parameter values: $g_0=0.1, g_1=0.15, c=0.6, s_1=1, s_2=0.2, k=0.1, r_1=r_2=1$, and $d_0=0.4$, where $r_1^2<4r_2d_0$ holds. We observe that for both initial points, (0.5, 0.3, 0.016) and (0.5, 0.18, 0.016), the trajectory of the system converges to the fixed point (2/3, 0, 0).}}\label{fig.1}
   \end{figure}

\begin{figure} [h!]
    \centering
    \includegraphics[width=0.8\textwidth]{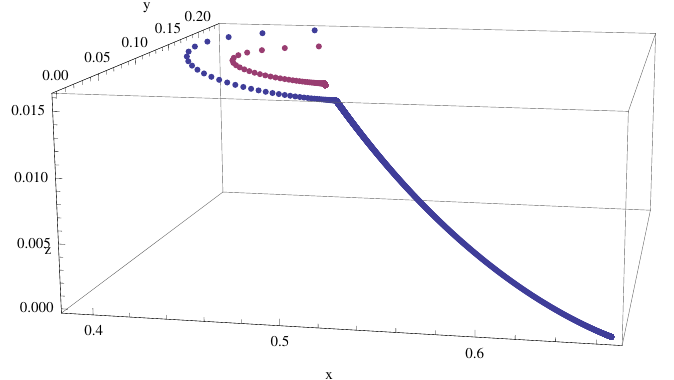}
    \caption{\footnotesize{ Considering the system defined by equation (\ref{systema}) with the following parameter values: $g_0=0.1, g_1=0.15, c=0.6, s_1=1, s_2=0.2, k=0.1, r_1=r_2=1$, and $d_0=0.25$, where $r_1^2=4r_2d_0$ holds.We observe the following behavior: For the initial point (0.5, 0.3, 0.016): The trajectory of the system converges to the fixed point (2/3, 0, 0). For the initial point (0.5, 0.18, 0.016): The trajectory of the system converges to the fixed point (1/2, 1/12, 3/200).}}\label{fig.2}
\end{figure}
\end{center}

\subsection{\bf Biological interpretation}\

(\ref{systema}) model represents the dynamics of glucose, insulin, and $\beta$-cell mass. In this section, we predict the normal behavior of the glucose regulatory system and the pathways leading to diabetes, depending on the parameter values and initial conditions.

Each point (vector) $(G,I,\beta)=(x,y,z)\in \Omega$ can be considered as a state (a measure) of glucose, insulin, and $\beta$-cell mass.
It can be seen from Proposition \ref{fixed} that system (\ref{systema}) has two fixed points $u^*_1$ and $u^*_2$. The fixed point $u^*_1$ corresponds to the disease pathology in the model, while $u^*_2$ represents a physiological state with healthy glucose levels. $d_0$ is the death rate at zero glucose.

Let us give some interpretations of our main results:
\begin{itemize}
  \item Assume that (\ref{parametr}) holds and let $r_1^2<4r_2d_0$. Under this condition on $d_0$ (i.e. the death rate at zero glucose), the value of $\beta$-cell mass and the level of insulin decrease, and the level of glucose increases. In summary, regardless of the initial state $(G,I,\beta)$ of the system (\ref{systema}), it converges to a fixed point, $u^*_1$, which represents a pathological condition (Interpretation of  Theorem \ref{Omega}).
\end{itemize}

\begin{itemize}
  \item (Case Theorem \ref{Omega-}, part \textbf{i}.) Assume that (\ref{parametr}) holds and let $r_1^2=4r_2d_0$. In addition, even if the glucose level is always greater than $g_0/g_1$, the system approaches a pathological state.
  \item (Case Theorem \ref{Omega-}, part \textbf{ii}.)If the value of $\beta$ mass in the initial state is less than $z^*$, then the system approaches the pathological state.
  \item (Case Theorem \ref{Omega-}, part \textbf{iii}.) If the value of $\beta$ mass is always greater than $z^*$, then the system approaches the physiological state.
\end{itemize}

\end{document}